\documentclass{article}
\usepackage{graphicx} 
\usepackage{amsmath,amssymb,mathrsfs}
\usepackage{cite}

\newcommand{\numberset}{\mathbb}

\newcommand{\Ll}{\mathcal{L}}

\newcommand{\F}{\numberset{F}}

\newcommand{\Xx}{\mathcal{X}}
\newcommand{\Yy}{\mathcal{Y}}

\newcommand{\Bb}{\mathcal{B}}
\newcommand{\Rr}{\mathcal{R}}
\newcommand{\Ee}{\mathcal{E}}
\newcommand{\Hh}{\mathcal{H}}

\usepackage{lmodern}
\usepackage[T1]{fontenc}
\usepackage{fancyhdr}
\usepackage{amssymb}
\usepackage{latexsym}
\usepackage{mathtools}
\usepackage{pdfpages}
\usepackage{sidecap}
\usepackage{braket}
\usepackage{amsmath}
\usepackage{amsthm}
\usepackage{graphicx}
\usepackage{subfig}
\usepackage{amsmath}
\usepackage{amssymb}
\usepackage{mathrsfs}
\usepackage{multirow}

\theoremstyle{plain}
\newtheorem{thm}{Theorem}[section]
\newtheorem{cor}[thm]{Corollary}

\newtheorem{prop}[thm]{Proposition}

\theoremstyle{definition}

\theoremstyle{remark}
\newtheorem{oss}[thm]{Remark}

\theoremstyle{definition}

\theoremstyle{definition}

\usepackage[a4paper, left=3.5cm, right=2.5cm, top=3cm, bottom=2.5cm]{geometry}
\usepackage{setspace}
\usepackage{tcolorbox}
\usepackage{changepage}

\title{AG codes from the Hermitian curve for Cross-Subspace Alignment in Private Information Retrieval}
\author{Francesco Ghiandoni, Massimo Giulietti, Enrico Mezzano, Marco Timpanella}
\date{}

\begin{document}

\maketitle

\begin{abstract}
Private information retrieval (PIR) addresses the problem of retrieving a desired message from distributed databases without revealing which message is being requested. Recent works have shown that cross-subspace alignment (CSA) codes constructed from algebraic geometry (AG) codes on high-genus curves can improve PIR rates over classical constructions. In this paper, we propose a new PIR scheme based on AG codes from the Hermitian curve, a well-known example of an $\F_\ell$-maximal curve, that is, a curve defined over the finite field with $\ell$ elements which attains the Hasse-Weil upper bound on the number of its $\mathbb F_\ell$-rational points. The large number of rational points enables longer code constructions, leading to higher retrieval rates than schemes based on genus 0, genus 1, and hyperelliptic curves of arbitrary genus. Our results highlight the potential of maximal curves as a natural source of efficient PIR constructions.
\end{abstract}


\section{Introduction}

Cryptography traditionally focuses on securing communication between users over insecure channels, ensuring that transmitted data remains confidential and inaccessible to unauthorized parties. When one of the parties is a server storing sensitive information, and there is a need to download certain data, the traditional security goal has been to protect the database’s content from potentially curious or malicious users \cite{adam1989security,dobkin1979secure}.

In 1995, as described in \cite{PIR1}, the reverse privacy problem was formulated for the first time: protecting the privacy of the user from the database itself. This led to the definition of the Private Information Retrieval (PIR) problem, which seeks the most efficient way for a user to retrieve a desired message from a set of distributed databases—each storing all the messages—without revealing to any single database which specific message is being retrieved.

From the outset, PIR schemes were observed to have close connections, with both analogies and differences, to several other cryptographic problems, including Blind Interference Alignment \cite{
PIR2,PIR3,
BIA2}, oblivious transfer \cite{OT}, instance hiding \cite{InstanceHiding1,InstanceHiding2}, multiparty computation \cite{chen2021gcsa}, secret sharing schemes \cite{SSS1,SSS2}, and locally decodable codes \cite{sun2020capacity}.

A trivial solution to the PIR problem is for the user to hide their interest by requesting the entirety of the database. However, this is highly inefficient, and the aim of PIR research is to design protocols that are as efficient as possible. In this context, the rate of a PIR scheme \cite{costruzione0} is defined as the ratio between the size of the desired information and the total amount of data downloaded to obtain it. The maximum achievable rate is known as the capacity of the PIR scheme \cite{PIR4,PIR5}, and the primary goal is to construct schemes approaching this theoretical bound.

The connection between PIR schemes and coding theory has been strong and influential from the very beginning. Early schemes used constructions based on covering codes \cite{PIR1,PIR8} and coding techniques applied to the database rather than mere replication \cite{PIR1,PIR5,PIR9}.

In recent years, Cross-Subspace Alignment (CSA) codes have been proposed as a method for building secure and private PIR schemes \cite{jia2020x,costruzione0}. The original construction in \cite{costruzione0}, which used Reed–Solomon codes, was later reinterpreted in an algebraic–geometric framework through AG codes \cite{costruzione1}. As noted in \cite{costruzione1}, this algebraic–geometric perspective enables PIR schemes to be built from curves of genus greater than zero, offering new trade-offs between the field size, the number of possible colluding servers, and the total number of servers. Explicit constructions from elliptic curves (genus 1) and hyperelliptic curves of arbitrary genus have been shown to achieve PIR rates surpassing those of the original CSA codes for a fixed field size. This improvement stems from the fact that increasing the genus yields curves with more rational points, enabling longer and more efficient code constructions. Consequently, curves with many rational points are natural candidates for CSA codes.

Recall that a projective, absolutely irreducible, non-singular, algebraic curve $\mathcal{X}$ over the finite field with $\ell$ elements $\F_\ell$ is called $\F_\ell$\emph{-maximal} if the number $|\mathcal{X}(\F_\ell)|$ of $\F_\ell$-rational points reaches the Hasse–Weil upper bound:
$$
|\mathcal{X}(\F_\ell)|= \ell + 1 + 2g(\mathcal{X})\sqrt{\ell},
$$
where $g(\mathcal{X})$ is the genus of the curve. The Hermitian, Suzuki, and Ree curves—collectively known as Deligne–Lusztig curves—are well-known examples of maximal curves. AG codes from these curves have been widely explored in recent literature for their excellent error-correcting performance \cite{korchmaros2019codes,landi2024two,matthews2004codes,timpanella2024generalization}.

Motivated by these observations, in this paper we propose a new PIR construction based on AG codes from the Hermitian curve. Our approach achieves significantly higher rates than previous schemes based on genus $0$ and genus $1$ curves \cite{costruzione1}, as well as hyperelliptic curves of arbitrary genus \cite{costruzione2}.

The paper is organized as follows. In Section \ref{sec prelim} we introduce the basic notions about algebraic curves over finite fields, Riemann Roch spaces, AG codes and their properties. Then, the concept of $X$-secure and $T$-private information retrieval (XSTPIR) is introduced and the costructions of XSTPIR schemes involving  evaluation codes over algebraic curves of genus $0,1$ (Theorems \ref{thm rate curva raz} and \ref{thm rate curva ellittica}) and on hyperelliptic curves of arbitrary genus (Theorem \ref{thm pir rate hyperelliptic}) are stated. In Section \ref{sezionecostruzione} the explixit costruction of a XSTPIR scheme by using AG codes over the Hermitian curve is exihibited, and the largest attainable PIR rate  is determined. Finally, in Section \ref{sec rate comparison}, the largest PIR rate (for a fixed field
size) for each of the four costructions involving algebraic curves is computed and a comparative analysis of these rates is provided.

\section{Preliminaries}\label{sec prelim}

\subsection{Background on algebraic curves and AG codes}
Throughout the paper, let $q = p^h$, for $p$ a prime number and $h>0$ an integer. Let $\mathcal{X}$ be a projective absolutely irreducible non-singular algebraic curve defined over the finite field $\mathbb{F}_q$ of genus $g(\mathcal{X})$. In the following, we adopt standard notation and terminology; see for instance \cite{HKT,stichtenoth2009algebraic}. In particular,  $\mathbb{F}_q(\mathcal{X})$ and $\mathcal{X}(\mathbb{F}_q)$ denote the field of $\mathbb{F}_q$-rational functions on $\mathcal{X}$ and the set of $\mathbb{F}_q$-rational points of $\mathcal{X}$, respectively, and ${\rm{Div}}(\mathcal{X})$ denotes the set of divisors of $\mathcal{X}$, where a divisor $D\in {\rm{Div}}(\mathcal{X})$ is a formal sum $n_1P_1+\cdots+n_rP_r$, with $P_i \in \mathcal{X}$, $n_i \in \mathbb{Z}$ and $P_i\neq P_j$ if $i\neq j$.
 The support $\mbox{\rm supp}(D)$ of the divisor $D$ is the set of points $P_i$ such that $n_i\neq 0$, while $\deg(D)=\sum_i n_i$ is the degree of $D$. 
We will say that the divisor $D$ is $\mathbb F_q$-rational if $n_i\neq 0$ implies $P_i\in \mathcal{X}(\mathbb{F}_q)$.
For a function $f \in \mathbb{F}_q(\mathcal{X})$, $(f)_0$ and $(f)_{\infty}$ are zero divisor of $f$ and the pole divisor of $f$, respectively, and $(f)=(f)_0-(f)_{\infty}$ is the principal divisor of $f$.
Two divisors $D$ and $D'$ are equivalent if they differ for a principal divisor; in this case we write $D\sim D'$.

The Riemann-Roch space associated with an $\mathbb F_q$-rational divisor $D$ is
$$\mathcal{L}(D) := \{ f \in \mathbb{F}_q(\mathcal{X}) \ : \ (f)+D \geq 0\}\cup \{0\},$$
and its dimension as a vector space over $\mathbb{F}_q$ is denoted by $\ell(D)$.  We recall some useful properties of Riemann-Roch spaces: 
\begin{itemize}
    \item{if $D'\sim D$ then $\Ll(D)$ and $\Ll(D')$ are isomorphic as $\F_q$-vector spaces and an isomorphism from $\Ll(D)$ in $\Ll(D')$ is given by $h\mapsto hf$, where $f$ is such that $D'= D+(f)$;}
    \item{if $D \le D'$ then $\Ll(D) \subseteq  \Ll(D')$;}
    \item{if $\deg(D)<0$, then $\Ll(D)=\{0\}$;}
     \item{$\Ll(D) \cdot\Ll(D') \subseteq \Ll(D+D')$, where $\Ll(D) \cdot\Ll(D')=\text{span}_{\mathbb{F}_q}\{f\cdot g \,:\, f\in \Ll(D), g\in \Ll(D')\}$ }.
\end{itemize}

Now, fix a set of pairwise distinct $\mathbb{F}_q$-rational points $\{P_1,\cdots,P_N\}$, and let $D=P_1+\cdots+P_N$. Take another $\mathbb F_q$-rational divisor $G$ whose support is disjoint from the support of $D$. The AG code $\mathcal{C}(D,G)$ is the (linear) subspace of $\mathbb{F}_q^N$ which is defined as the image of the evaluation map $ev :  \mathcal{L}(G) \to \mathbb{F}_q^N$ given by $ev(f) = (f(P_1),f(P_2) ,\ldots,f(P_N))$. 
In particular $\mathcal{C}(D,G)$ has length $N$ (see \cite[Theorem 2.2.2]{stichtenoth2009algebraic}. Moreover, if $N>\deg(G)$ then $ev$ is an embedding and $\ell(G)$ equals the dimension $k$ of $\mathcal{C}(D,G)$. The minimum distance $d$ of $\mathcal{C}(D,G)$, usually depends on the choice of $D$ and $G$. A lower bound for $d$ is $ d^*=N-\deg(G)$, where $d^*$ is called the Goppa designed minimum distance of $\mathcal{C}(D,G)$. Furthermore, if $\deg(G)>2g(\Xx)-2$ then $k=\deg(G)-g(\Xx)+1$ by the Riemann-Roch Theorem; see \cite[Theorem 2.65]{HLP} and \cite[Corollary 2.2.3]{stichtenoth2009algebraic}. The (Euclidean) dual code $\mathcal{C}(D,G)^\perp$ is again an AG code, i.e., $\mathcal{C}(D,G)^\perp=\mathcal{C}(D,G^*)$ for some $G^* \in \textnormal{Div}(\mathcal{X}),$ see \cite[Proposition 2.2.10]{stichtenoth2009algebraic}. It has parameters $[N^\perp,k^\perp,d^\perp]_q,$ where $N^\perp=N,$ $k^\perp=N-k$ and $d^\perp \ge \deg(G)-2g(\Xx)+2.$

Two significant families of algebraic curves that will be considered in the rest of the paper are those of the Hermitian curves and the hyperelliptic curves.
The Hermitian curve $\mathcal{H}_q$ is defined as any $\mathbb{F}_{q^2}$-rational curve projectively equivalent to the plane curve with affine equation
\begin{equation}\label{normtrace}
X^{q+1}=Y^q+Y.
\end{equation}
This is a non-singular curve, and its genus is $g(\mathcal{H}_q)=q(q-1)/2$. The curve $\mathcal{H}_q$ has a unique point at infinity $P_{\infty}=(0:0:1)$ and a total of $|\mathcal{H}_q(\mathbb{F}_{q^2})|=q^3+1$ $\mathbb{F}_{q^2}$-rational points. Therefore, $\mathcal{H}_q$ is an $\mathbb{F}_{q^2}$-maximal curve.
The Riemann-Roch spaces associated to the divisors $mP_{\infty}$ of $\mathcal{H}_q$ are well-known, as the next result shows.
\begin{prop} \label{prop base RR herm 1-point}
    Let $m>0$, a basis for $\Ll(mP_{\infty})$ is 
    \begin{equation*}
        \{ x^i y^j : iq+j(q+1) \le m, \ \  i \ge 0, \ \  0 \le j \le q-1 \}.
    \end{equation*}
\end{prop}

 A hyperelliptic curve $\mathcal{Y}$ of genus $g \geq 2$ over a field $\mathbb{F}_q$ is an algebraic curve given by an affine equation of the form $ y^2 + H(x)y = F(x)$, where $F(x)$ is a monic polynomial of degree $2g+1$ or $2g+2$, and $\deg(H) \leq g$. If $p\neq 2$, one can assume $H(x) = 0$, simplifying the equation to $y^2 = F(x)$.
It is readily seen that for such curves \begin{equation} \label{eq uppbound for hyp curves 2q+1}
    \#\mathcal{Y}(\F_q) \le 2q+1.\end{equation}

\subsection{Background on Private Information Retrieval schemes}
In this section, we introduce the PIR problem and the approach that we will use throughout the paper to construct PIR schemes. For a precise information-theoretic formulation of secure and private information retrieval, see~\cite{costruzione0}.

The PIR problem can be described as follows. We have a set of $N$ non-communicating databases each containing $M$ independent files $s_1, \ldots, s_M \in \mathbb{F}_q^L$. We refer to a single coordinate of such a vector as a \emph{fragment} of a file. 
In Private Information Retrieval (PIR), we want to retrieve one of these files $s_\mu = (s_{\mu,1}, \ldots, s_{\mu,L})$ without revealing the desired index $\mu \in [M]=\{1,\ldots,M\}$ to the servers.  To do so, the user generates $N$ queries $Q_1, \ldots, Q_N$ and sends $Q_i$, $i\in \{1, 2, \ldots, N\}$ to the $i$-th database. After receiving query $Q_i$, the $i$-th database returns an answer $A_i$ to the user. The user must be able to obtain the desired message $s_\mu$ from all the answers $A_1, \ldots, A_N$. For the precise nature of how the data is stored, how we query each server, and the method we use to recover the desired file from the responses we refer to \cite{costruzione2,costruzione1}.
To have \emph{data security}, we require that the data stored at any $X$ servers reveals nothing about the file contents. Similarly, to have \emph{query privacy}, we require that the queries sent to any $T$ servers reveal nothing about the desired file index.

An $X$-secure and $T$-private information retrieval (XSTPIR) is a private information retrieval scheme that guarantees data security against collusion among up to $X$ servers and ensures user privacy against collusion among up to $T$ servers. 
The parameters $T$ and $X$ can be chosen arbitrarily based on the relative importance of security and privacy for a given application.

In order to achieve privacy and security, an appropriate noise into the queries and data is introduced. In particular, in \cite{costruzione1} the following result is proved. 

\begin{thm}[\cite{costruzione1}, Theorem 2.1]\label{thm costr 1 decomp spazio}
    Let $A$ be an algebra over $\F_q$ and $V^{enc}_l,V^{sec}_l, V^{query}_l,V^{priv}_l$ be finite-dimensional subspaces of $A$ for any $l\in[L]$. Moreover, let $\varphi:A \to \F_q^N$ be an $\F_q$-algebra homomorphism such that:
    \begin{enumerate}
        \item{$V^{query}_l = \text{span}\{ h_l\}$, where $h_l$ is a unit in $A,$}
        \item{$\text{dim}(\sum_{l=1}^L V^{info}_l)= \sum_{l=1}^L\text{dim}(V^{info}_l),  $ where $V_l^{info}:=V_l^{enc}\cdot V_l^{query},$}
        \item{$V^{info} \cap V^{noise} = \{ 0 \},$ where $V^{info}:=\bigoplus_{l=1}^LV_l^{info}$ and $V^{noise}:=\sum_{l=1}^L(V_l^{enc}\cdot V_l^{priv}+V_l^{sec}\cdot V_l^{query}+V_l^{sec}\cdot V_l^{priv}),$}
        \item{$\varphi$ is injective over $V^{info} \oplus V^{noise} .$}
    \end{enumerate}   
 Then there exists a $X$-secure and $T$-private information retrieval scheme with rate $R=\frac{L}{N}$, where $$X\leq d_{sec}-1, \quad T\leq d_{priv}-1,$$ and
 $$d_{sec}:=\min\{d^\perp (\varphi(V_l^{sec})):l \in [L]\}, \quad d_{priv}:=\min\{d^\perp (\varphi(V_l^{priv})):l \in [L]\}.$$
\end{thm}

In \cite{costruzione1}, the vector spaces satisfying assumptions 1)–4) of Theorem \ref{thm costr 1 decomp spazio} are explicitly constructed as Riemann–Roch spaces associated with divisors on algebraic curves of genus $0$ and $1$. In what follows, we present the XSTPIR schemes derived from these constructions: see Theorem \ref{thm rate curva raz} for the case of rational curves and Theorem \ref{thm rate curva ellittica} for the case of elliptic curves.

\begin{thm}[\cite{costruzione1}] \label{thm rate curva raz}
    Let 
    $X$ and $T$ be some fixed security and privacy parameters. 
    If
    $2L+X+T\le q$, then for
$N=L+X+T$ 
there exists a $X$-secure and $T$-private PIR scheme with rate $$\Rr^\mathcal{X}=\frac{L}{N}=1-\frac{X+T}{N}.$$
\end{thm}

\begin{thm}[\cite{costruzione1}] \label{thm rate curva ellittica}
    Let $\mathcal{E}$ be an elliptic curve defined over $\F_q$, $p>3$, by the affine equation $y^2-f(x)=0$ where $\deg(f)=3$, and let $X$ and $T$ be some fixed security and privacy parameters. Let $L$ be an odd integer, $J=\frac{L+1}{2}$, and $\{P_1,\bar {P_1},\dots,P_{J}, \bar {P_J}\}$ be a set of $2J$ $\F_q$-rational points of $\mathcal{E}$ distinct from $P_\infty$ and not lying on the line $y=0,$ where  $P_j=(\alpha_j,\beta_j)$  and $\bar P_j=(\alpha_j,-\beta_j),$  $j\in [J].$
    If there exist $N+1=L+X+T+9$ $\F_q$-rational points of $\mathcal{E}$ distinct from $P_{\infty},P_1,\ldots,P_J, \bar{P}_1,\ldots, \bar{P}_J$ and not lying on $y=0$, 
    then there exists a $X$-secure and $T$-private PIR scheme with rate $$\Rr^{\mathcal{E}}=\frac{L}{N}=1-\frac{X+T+8}{N}.$$ 
\end{thm}
As observed in  \cite{costruzione1}, even if $\Rr^\mathcal{X} > \Rr^\mathcal{\Ee}$ for  fixed $X, T, N,$
the above constructions might work for different field sizes. This follows from the fact that elliptic curves allow for more rational points compared to curves of genus $0.$  This opens up the possibility of leveraging a larger number of servers and boost the rate of the genus-one scheme when operating within a fixed field size constraint. Actually, in \cite{costruzione1} the authors show that for a fixed field size and for $X,T$ sufficiently large that the PIR schemes from elliptic curves offer higher rates than those from rational curves.

Extending the results of \cite{costruzione1}, XSTPIR schemes from hyperelliptic curves of arbirtrary genus have been  constructed 
 in \cite{costruzione2},
We state below the main result from \cite{costruzione2}.
\begin{thm}[\cite{costruzione2}, Theorem VI.1]\label{thm pir rate hyperelliptic}
    Let $\Yy$  be a hyperelliptic curve of genus $g$ defined over $\mathbb{F}_q$ by the affine equation $y^2 +H(x)y = F(x)$.
Let $L \geq g,$  $L \equiv  g$ (mod 2) and $J =
\frac{L+g}{2}.$  Let $\lambda_1, . . . , \lambda_J \in \F_q$ be such that $F(\lambda_j ) \neq 0$ and set $h =
\prod_{j\in[J]}(x-\lambda_j ).$
Let $N = L + X + T + 6g + 2$ for some security and privacy parameters $X$ and $T.$ If $\mathcal{Y}$ has $N + g$ rational
points $P_1,\dots, P_{N+g}$ different from $P_\infty$, not lying on $y=0$, and such that $h(P_n)\neq 0,$ then there exists a PIR scheme
with rate $$R^{\mathcal{Y}} = 1 -\frac{X + T + 6g + 2}{N},$$
which is $X$-secure and $T$-private.
\end{thm}

\section{XSTPIR schemes from the Hermitian curve}\label{sezionecostruzione}
The aim of this section is to consider Riemann-Roch spaces of the Hermitian curve $\Hh_q$ that satisfy the assumptions of Theorem \ref{thm costr 1 decomp spazio}, in order to obtain new XSTPIR schemes.

To do so, let $m\in \{q-1,\ldots, q^2-1\}$ and $L:=mq-g(\Hh_q)=q(m-\frac{q-1}{2})$.
Throughout this section, for any $i\in [m]$, choose $m$ distinct $\alpha_i\in\mathbb{F}_{q^2}\setminus \{0\}$ and  let $P_{i,z}= (\alpha_i, \beta_{i,z})$, for $z\in [q]$, be $mq$ affine $\mathbb{F}_{q^2}$-rational points of $\mathcal{H}_q$. Then, define $h \in \F_{q^2}(x)$ by
\begin{equation*}
       h:= \prod_{i=1}^{m} \frac{1}{(x-\alpha_i)},
\end{equation*}
whence
\begin{equation*}
    (h) = mq P_{\infty} - \sum_{i=1}^{m} \sum _{j=1}^{q} P_{i,j}.
\end{equation*}
We consider the following divisor on $\mathcal{H}_q$.
\begin{equation} \label{eq espr Dinfo}
    D^{info}:=\sum_{i=1}^{m} \sum _{z=1}^{q} P_{i,z} - P_{\infty}=-(h)+(L+g(\Hh_q)-1)P_{\infty}.
\end{equation}
For our construction, we aim to construct an $\F_{q^2}$-basis of $\mathcal{L}(D^{info})$.
To do so, first we find a $\F_{q^2}$-basis of $\mathcal{L}((L+g(\Hh_q)-1)P_\infty)$; then, multiplication by $h$ will yield an $\F_{q^2}$-basis of $\mathcal{L}(D^{info})$.

For every $z\in [q]$ and $i\in [m-z+1]$, let 
\begin{equation} \label{eq hiz}
  \bar h^{(z)}_i := y^{z-1} \cdot \prod_{i' \in [m-z+1]\setminus \{i\}} (x-\alpha_{i'}).
\end{equation}
Also, let $\bar \Bb = \{\bar h_1^{(1)},\dots,\bar h_m^{(1)}, \bar h_1^{(2)},\dots,\bar h_{m-1}^{(2)},\dots, \bar h_1^{(q)},\dots,\bar h_{m-(q-1)}^{(q)} \}$.
\begin{prop} \label{propbase}
The set $\bar \Bb$ is an $\F_{q^2}$-basis of $\mathcal{L}((L+g(\Hh_q)-1)P_{\infty}).$
\end{prop}
\begin{proof}
  Since $L\geq q,$ the Riemann-Roch theorem yields \begin{equation*}
\ell((L+g(\Hh_q)-1)P_{\infty})=\ell((mq-1)P_{\infty})=\text{deg}{((mq-1)P_{\infty}})+1-g(\Hh_q) = mq-g(\Hh_q)=L.
\end{equation*}
    Moreover, $$ \#\bar \Bb=m+(m-1)+(m-2)+\dots+(m-q+1)=mq- \frac{(q-1)q}{2}=mq-g(\Hh_q)=L,$$ 
    so $\bar \Bb$ consists of exactly $L$ functions.
   As $$v_{P_\infty}(\bar h^{(z)}_i)=-(m-z)q-(z-1)(q+1)=-(mq-1)+q-z,$$ we have that $\bar h^{(z)}_i \in \Ll((mq-1)P_{\infty})$ for any $z\in [q]$ and $i\in [m-z+1],$ and hence it remains to prove that the functions in $\bar \Bb $ are $\F_{q^2}$-linearly independent.
To do so, assume there exist $\gamma_{i,z} \in \F_{q^2}$,  $z\in [q]$ and $i\in [m-z+1],$ such that
\begin{equation*} 
      \sum_{z=1}^q \sum_{i=1}^{m-(z-1)} \bar h_i^{(z)} \gamma_{i,z} = 0,
\end{equation*}
i.e.
\begin{equation}\label{eq lin indip of interp bas}
      \sum_{z=1}^q y^{z-1}  \cdot  \sum_{i=1}^{m-(z-1)} \gamma_{i,z} \cdot \prod_{i' \in [m-(z-1)]\setminus \{i\}} (x-\alpha_{i'})=0.
\end{equation}
Now, let 
$$
f_z(x):=\sum_{i=1}^{m-(z-1)} \gamma_{i,z} \cdot \prod_{i' \in [m-(z-1)]\setminus \{i\}} (x-\alpha_{i'}),
$$
and observe that for every $z \in [q]$ and $i \in [m-z+1],$  we have
\begin{equation}\label{fz}
f_{z}(\alpha_i)=\gamma_{i,z} \cdot \prod_{i' \in [m-(z-1)]\setminus \{i\}} (\alpha_i-\alpha_{i'})=\gamma_{i,z} \cdot K_{i,z}, 
\end{equation}
with $K_{i,z} \in \F_{q^2}\setminus \{ 0\}.$    
For a fixed $i\in [m]$, by evaluating \eqref{eq lin indip of interp bas} at the points $P_{i,z}=(\alpha_i,\beta_{i,z})$, where $z\in [q]$, we obtain  \begin{equation}\label{sist vandermonde}
\begin{cases}
    f_1(\alpha_i)+\beta_{i,1}f_2(\alpha_i)+\dots+\beta_{i,1}^{q-1}f_q(\alpha_i)=0 \\
    f_1(\alpha_i)+\beta_{i,2}f_2(\alpha_i)+\dots+\beta_{i,2}^{q-1}f_q(\alpha_i)=0 \\
    \vdots \\
    f_1(\alpha_i)+\beta_{i,q}f_2(\alpha_i)+\dots+\beta_{i,q}^{q-1}f_q(\alpha_i)=0.
\end{cases}
\end{equation} 
This can be viewed as a linear system in unknowns $f_1(\alpha_i), \dots, f_q(\alpha_i)$, whose associated matrix is Vandermonde-type with respect to pairwise distinct coefficients $\beta_{i,1},\dots,\beta_{i,q}$. Therefore,  \eqref{sist vandermonde} admits only the trivial solution. Therefore, $f_z(\alpha_i)=0$ for all $i\in [m]$ and $z\in [q]$. In particular, for any $z \in [q]$ and $i \in [m-z+1],$ this together with \eqref{fz} implies $K_{i,z}\gamma_{i,z}=0$;
  since $K_{i,z}$ is non-zero, we obtain $\gamma_{i,z}=0,$ which concludes the proof.
\end{proof}

As a by-product of Proposition \ref{propbase} we have the following.
\begin{cor} \label{cor base LDinfo} The set 
    \begin{equation*}
    \Bb = \{h\cdot \bar h_1^{(1)},\dots,h\cdot \bar h_m^{(1)}, h\cdot \bar h_1^{(2)},\dots,h\cdot \bar h_{m-1}^{(2)},\dots, h\cdot \bar h_1^{(q)},\dots,h\cdot \bar h_{m-(q-1)}^{(q)} \} 
\end{equation*}
is an $\F_{q^2}$-basis of $\mathcal{L}(D^{info})$.
\end{cor}

Now note that from \eqref{eq hiz} it follows that 
\begin{equation} \label{eq elti base Dinfo}
   h^{(z)}_i =h\cdot \bar h^{(z)}_i = \frac{y^{z-1}}{(x- \alpha_i)(x-\alpha_{m-(z-2)})\cdots(x-\alpha_{m-1})(x-\alpha_m)},
\end{equation}
whence, being $(y)=(q+1)P_0-(q+1)P_\infty$, we obtain
\begin{equation*}
   (h^{(z)}_i) = (q-z+1)P_{\infty} + (z-1)(q+1)P_0- \sum_{j=1}^q \left( P_{i,j}+  P_{m-(z-2),j} +\dots+  P_{m-1,j} +  P_{m,j}\right).
\end{equation*}
Thus,  we can estimate the principal divisors of $h^{(z)}_i $  independently of $z$ and $i$:
\begin{equation} \label{eq upperbound (hl)}
   (h^{(z)}_i ) \le q P_{\infty} + (q^2-1)P_0. 
\end{equation}

From now on, in order to simplify the notations, we will denote $\Bb = \{h_1\,\dots,h_L\}$. For the explicit scheme construction,  we consider the following divisors on $\mathcal{H}_q$ and the corresponding Riemann-Roch spaces, namely
\begin{eqnarray}
D^{enc}_l = 0,  \ \ &&\ \ V^{enc}_l= \mathcal{L}(D^{enc})=\F_{q^2}, \label{eq DVenc}\\
D^{sec}_l = (X+2g(\Hh_q)-1)P_{\infty} + (h_l), && \ \ V^{sec}_l= \mathcal{L}(D^{sec}),  \label{eq DVsec} \\
D^{query}_l =  -(h_l),   && \ \  V^{query}_l= \mathcal{L}(D^{query})=\textnormal{span}\{h_l\},  \label{eq DVquery}
\\
D^{priv}_l =  (T+2g(\Hh_q)-1)P_{\infty}, && \ \  V^{priv}_l= \Ll(D^{priv}). \label{eq DVpriv}
\end{eqnarray}
Observe that the Riemann-Roch Theorem yields $\ell(D^{sec})=X+g(\Hh_q)$ and $\ell(D^{priv})=T+g(\Hh_q).$ 
Moreover, for $l \in [L]$ we obtain
\begin{equation} \label{eq V l info}
V^{info}_l:= V^{enc}_l \cdot V^{query}_l = \F_{q^2} \cdot \mathcal{L}(-(h_l))= \mathcal{L}(-(h_l)) =\text{span}\{h_l\},
\end{equation} 
and hence by setting $D_l^{info}:=-(h_l)$ we have $V_l^{info}=\mathcal{L}(D_l^{info}).$ Moreover $\mathcal{L}(D_1^{info})\oplus \cdots \oplus \mathcal{L}(D_L^{info})=\textnormal{span}\{h_1,\dots,h_L\}=\mathcal{L}(D^{info})=:V^{info}$ (see Corollary \ref{cor base LDinfo}).
In order to apply Theorem \ref{thm costr 1 decomp spazio} we also have to consider
\begin{equation} \label{eq Vnoise generale}
V^{noise} := \sum _{l=1}^L (V^{enc}_l \cdot V^{priv}_l + V^{sec}_l \cdot V^{query}_l + V^{sec}_l \cdot V^{priv}_l).
\end{equation}
As pointed out in \cite[Section IV]{costruzione1}, one way to take advantage of Theorem \ref{thm costr 1 decomp spazio} is now to find two divisors $D^{noise}$ and $D^{full}$ on $\mathcal{H}_q$ such that $V^{noise}\subseteq \mathcal{L}(D^{noise})$, $D^{info},D^{noise}\leq D^{full}$ and $\mathcal{L}(D^{info}) \cap \mathcal{L}(D^{noise}) = \{ 0 \}.$ Moreover, in order to minimize the dimension of $\mathcal{L}(D^{noise})$ and prevent it from growing with $L,$ the divisor $D^{noise}$ should have minimal degree and be independent of $L$. 
With this aim, we introduce
\begin{equation} \label{eq Dnoise}
    D^{noise}= (X+T+4g(\Hh_q)+q-2)P_{\infty}  + (q^2-1)P_0,\end{equation}
and
\begin{equation} \label{eq D full}
    D^{full}= \sum_{i=1}^m \sum_{j=1}^q P_{i,j} + (X+T+4g(\Hh_q)+q-2)P_{\infty}  + (q^2-1)P_0.
\end{equation}

We are now in position to prove the following result.
\begin{thm} \label{thm PIR sheme for the hermitian}
    Let $\mathcal{H}_q$ be the Hermitian curve of genus $g(\mathcal{H}_ q)=\frac{q(q-1)}{2}.$ Let $g(\mathcal{H}_ q)\leq L\leq q^3-g(\mathcal{H}_ q),$ with $L+g(\mathcal{H}_ q)\equiv_q 0$ and set $m=\frac{L+g(\mathcal{H}_ q)}{q}.$ Let $P_{i,z}= (\alpha_i, \beta_{i,z})$, $i\in [m]$, $z\in [q]$, be $mq$ affine $\mathbb{F}_{q^2}$-rational points of $\mathcal{H}_q$ with $\alpha_i\neq 0$. If there exist $L+X+T+\frac{7q^2-3q-6}{2}+1$ $\F_{q^2}$-rational points of $\mathcal{H}_q$ distinct from $P_{\infty}, P_0$ and $P_{i,z}$, $i\in [m]$, $z\in [q]$, then there exists a $X$-secure and $T$-private PIR scheme with rate \begin{equation} \label{eq rate da costr hermitiana}
        \Rr^{\Hh_q} = \frac{L}{N}=1- \dfrac{X+T+3q^2-q-2}{N},\end{equation}
where $N=L+X+T+3q^2-q-2,$ and $X$ and $T$ are some fixed  security and privacy parameters.
\end{thm}
\begin{proof}
Consider the divisors (and the corresponding Riemann-Roch spaces) defined  by Equations \eqref{eq DVenc},\eqref{eq DVsec},\eqref{eq DVquery},\eqref{eq DVpriv} and \eqref{eq V l info}. In particular, let $D^{info}$, $D^{noise}$, and $D^{full}$ be defined by Equations \eqref{eq espr Dinfo}, \eqref{eq Dnoise}, and \eqref{eq D full}, respectively. By Equation \eqref{eq upperbound (hl)} and direct computations, it can be checked that
\begin{eqnarray}
D^{enc}_l + D^{priv}_l \!\!&=&\!\! 0 + (T+2g(\Hh_q)-1)P_{\infty} = (T+2g-1)P_{\infty}\leq D^{noise},\label{eq dnoise 1}
\\
D^{sec}_l + D^{query}_l \!\!&=& \!\!(X+2g(\Hh_q)-1)P_{\infty} + (h_l) -(h_l) = (X+2g(\Hh_q)-1)P_{\infty}\leq D^{noise},\label{eq dnoise 2}
\\
D^{sec}_l + D^{priv}_l \!\!&=& \!\!(X+T+4g(\Hh_q)-2)P_{\infty}  + (h_l) \leq D^{noise};\label{eq dnoise 3}
\end{eqnarray} so $V^{noise}\subseteq \mathcal{L}(D^{noise}),$ where $V^{noise}$ is defined as in Equation \eqref{eq Vnoise generale}.

We  observe that the functions in $\mathcal{L}(D^{info})$ have no poles at  $P_{\infty}$ or $P_0,$ while the functions in $\mathcal{L}(D^{noise})$ have a pole exclusively at either $P_{\infty}$ or $P_0,$ so  $\mathcal{L}(D^{info}) \cap \mathcal{L}(D^{noise}) = \{ 0 \}$.
It follows that
\begin{equation*}
    \mathcal{L}(D^{info})\oplus \mathcal{L}(D^{noise}) \subseteq \mathcal{L}(D^{full})
\end{equation*}

where $\text{deg}(D^{full})=mq + (X+T+4g(\Hh_q) +q-2)+(q^2-1)=L+X+T+\frac{7q^2-3q-6}{2}.$

Consider now the evaluation map $\bar{\varphi}$ defined on a set $\bar{\mathcal{P}}$ of $\deg(D^{full})+1$ $\F_{q^2}$-rational points of $\mathcal{H}_q$  disjoint from $\{P_0, P_\infty\}\cup \{P_{i,z}\}_{i,z}.$ Observe that such points exist by hypothesis. Then such a map is injective and we can define the associated AG code $\mathcal{C}(\bar {\mathcal{P}},D^{full})$, whose dimension will be \begin{eqnarray*}
     N := \ell(D^{full})\!\!&=&\!\!\deg(D^{full})-g(\Hh_q)+1=(L+X+T+5g(\Hh_q)+q-2+q^2-1)+1-g(\Hh_q)\\ \!\!&=&\!\!L+X+T+3q^2-q-2. 
 \end{eqnarray*}
 Therefore there exists a subset $\mathcal{P}\subset \bar {\mathcal{P}}$ of size $N$ (so called an information set of $\mathcal{C}(\bar {\mathcal{P}},D^{full})$) such that the evaluation map $\varphi = ev_\mathcal{P}: L(D^{full}) \to \F_{q^2}^N$ on $\mathcal{P}$ is injective. In particular, $\textnormal{ev}_\mathcal{P}$ is injective when restricted to the subspace $V^{info}\oplus V^{noise} \subseteq \mathcal{L}(D^{info})\oplus \mathcal{L}(D^{noise}).$ Therefore, Theorem \ref{thm costr 1 decomp spazio} applies, ensuring the existence of a $(\min\{d^\perp (\mathcal{C}(\mathcal{P},D_l^{sec}))-1: l \in [L]\})$-secure and $(\min\{d^\perp (\mathcal{C}(\mathcal{P},D_l^{priv}))-1: l \in [L]\})$-private information retrieval scheme with rate $\Rr^{\Hh_q}=\frac{L}{N}.$ Since for $l \in [L]$ \[ d^\perp (\mathcal{C}(\mathcal{P},D_l^{sec})) 
 \ge \ell (D_l^{sec})+1 \ge (X+g(\Hh_q))-g(\Hh_q)+1 = X+1\] and \[ d^\perp (\mathcal{C}(\mathcal{P},D_l^{priv})) 
 \ge \ell (D_l^{priv})+1 \ge (T+g(\Hh_q))-g(\Hh_q)+1 = T+1,\] the statement follows.

\end{proof}

\begin{oss}
    It is well known that a monomial $\F_{q^2}$-basis of $\mathcal{L}((L+g(\Hh_q)-1)P_\infty)$ is given by \[\{x^iy^j: iq+j(q+1) \leq L+g(\Hh_q)-1, 0 \leq j \leq q-1\},\] see Proposition \ref{prop base RR herm 1-point}.
    However, after multiplication by $h$, the corresponding basis functions of $\mathcal{L}(D^{info})$ would have a zero at $P_\infty$ of order growing with $L$, consequently forcing $\deg(D^{full})$ to grow with $L$ and essentially reducing the rate of the PIR scheme.
\end{oss}

The following proposition provides an explicit formula for the largest rate achieved by the construction in Theorem \ref{thm PIR sheme for the hermitian}.
\begin{prop} \label{prop rate migliore hermitiana}
    The maximum rate  of the XSTPIR schemes in Theorem \ref{thm PIR sheme for the hermitian} is given by $\Rr_{\max}^{\Hh_q}=\frac{L}{N},$ where  $L=mq-g(\Hh_q)$ and $m=\lfloor \frac{q^3-3q^2+q+1-(X+T)}{2q} \rfloor.$
\end{prop}

\begin{proof} 
In order to compute the largest rate in Equation \eqref{eq rate da costr hermitiana}, one needs to  maximize the parameter $L$ (and cons. $N$, where $N=L+X+T+\frac{7q^2-3q-6}{2}$). Therefore, since $\#\text{supp}(D^{full})=L+g(\Hh_q)+2,$ we search for the largest $L$ such that  $L+g(\Hh_q)\equiv_q 0$ and
 \begin{align*}
    \ \#\mathcal{H}_q(\F_{q^2})=q^3+1 &\ge \deg(D^{full})+1 + L+g(\Hh_q)+2 \\&\ge L+X+T+\frac{7q^2-3q-6}{2}+L+g(\Hh_q)+3 \\
    &= 2L+X+T+4q^2-2q.
\end{align*}
    By Theorem \ref{thm PIR sheme for the hermitian} and  $\#\Hh_q(\F_{q^2})=q^3+1$ we obtain the best rate $\Rr_{\max}^{\Hh_q}=\frac{L}{N}$ for $L=mq-g(\Hh_q)$ and $m=\lfloor \frac{q^3-3q^2+q+1-(X+T)}{2q} \rfloor.$ 

\end{proof}

\section{Rates comparison} \label{sec rate comparison}

The aim of this section, in accordance  with \cite[Section VI]{costruzione1} and \cite[Section VI.c]{costruzione2}, is to compute and compare the maximal rates $\Rr_{\max}$ of the four families of XSTPIR schemes from algebraic curves for some fixed field size and for security and privacy parameters $X$ and $T$, that is

\begin{itemize}
    \item [(a)] $\Rr^\mathcal{X}=1-\frac{X+T}{N},$ where $\mathcal{X}$ is a rational curve;
    \item [(b)] $\Rr^\mathcal{\Ee}=1-\frac{X+T+8}{N}$, where $\mathcal{E}$ is an elliptic curve;
    \item[(c)] $\Rr^{\mathcal{Y}} = 1 -\frac{X + T + 6g + 2}{N},$ 
    where $\mathcal{Y}$ is an hyperelliptic curve of genus $g$;
    \item[(d)] $\Rr^{\mathcal{H}_q} = 1-\frac{X+T+3q^2-q-2}{N},$  where $\mathcal{H}_q$ is the Hermitian curve of degree $q+1$;
\end{itemize}
see Theorems \ref{thm rate curva raz}, \ref{thm rate curva ellittica}, \ref{thm pir rate hyperelliptic}
 and \ref{thm PIR sheme for the hermitian}. Since the rate expression for each of the four constructions is given by $\Rr = \frac{L}{N}$, where $N-L$ depends only on $X$, $T$, $g$, and $q$, maximizing the rate of each construction requires finding the largest possible $L = L^{\max}$ (or equivalently, the largest possible $J = J^{\max}$) that satisfies the conditions stated in the corresponding theorems for the four constructions.
 We start by computing and comparing the largest rates $\Rr_{\max}^\mathcal{X}$ and $\Rr_{\max}^\mathcal{E}$ for cases $(a),(b).$
 \begin{prop} \label{prop msx pir rate g eq 0,1}
    The maximum rates of a XSTPIR scheme for cases $(a)$ and $(b)$ are 
\begin{itemize}
    \item $\Rr_{\max}^\mathcal{X}=\frac{L_{\max}^\mathcal{X}}{L_{\max}^\mathcal{X}+X+T}$ with $L_{\max}^\mathcal{X} =\lfloor \frac{q-(X+T)}{2} \rfloor.$ 
    \item $\Rr_{\max}^\mathcal{\Ee}=\frac{L_{\max}^\mathcal{E}}{L_{\max}^\mathcal{E}+X+T+8}$ with $L_{\max}^\mathcal{E}=2 \cdot \left\lfloor \frac{\# \Ee(\F_q)-(X+T+\gamma+9)}{4} \right\rfloor -1$
\end{itemize}
respectively.
 \end{prop}
 \begin{proof}
 In case $(a)$ the best rate $\Rr_{\max}^\mathcal{X}$ is achieved by considering the largest $L$ (and cons. the largest $N = L + X + T$) such that $q\geq L+X+T,$ i.e., $L_{\max}^\mathcal{X} =\lfloor \frac{q-(X+T)}{2} \rfloor.$ 
 
 In case $(b)$ the largest rate $\Rr_{\max}^\mathcal{\Ee}$ is obtained by computing the largest $L=2J-1$ (and cons. the largest $N=L+X+T+8$) such that
\begin{equation*}
    \# \Ee(\F_q) \ge N+2J+2+\gamma=2L+X+T+11+\gamma ,
\end{equation*}
where $\gamma=\#\{z \in \F_q: f(z)=0\}\le 3.$ Therefore \begin{equation}\label{eq L ellitt max}L_{\max}^\mathcal{E}=2 \cdot \left\lfloor \frac{\# \Ee(\F_q)-(X+T+\gamma+9)}{4} \right\rfloor -1\end{equation} and $N_{\max}^\mathcal{E}=L_{\max}^\mathcal{E}+X+T+8.$
\end{proof}
It follows that the highest rates for the case $(b)$ are obtained by maximizing $\# \Ee(\F_q)-\gamma,$ that is, by considering $\F_q$-maximal (or optimal) elliptic curves.
The following proposition corroborates the findings of \cite[Section VI]{costruzione1}.
\begin{prop} \label{prop ellitt 1 costr meglio di raz}
    Let $\mathcal{X}$ be a rational curve and $\Ee$ an elliptic curve both defined over $\mathbb{F}_q$. If  $$\ \#\Ee(\F_q)\geq q\left(1+\frac{8}{X+T}\right)+\gamma+7,$$  then $\ \Rr_{\max}^\mathcal{\Ee}> \Rr_{\max}^\mathcal{\Xx}.$ 
\end{prop}
\begin{proof}
    By Proposition \ref{prop msx pir rate g eq 0,1} and $z-1 \leq \lfloor z \rfloor \leq z, \ $ it follows that $\Rr_{\max}^\mathcal{\Ee}=\frac{L_{\max}^\mathcal{\Ee}}{N_{\max}^\mathcal{\Ee}} > \frac{L_{\max}^\mathcal{\Xx}}{N_{\max}^\mathcal{\Xx}}=\Rr_{\max}^\mathcal{\Xx}$ holds for 
    \[ \frac{q-(X+T)}{2} \cdot (X+T+8) \leq (X+T)\cdot  \frac{\#\Ee(\F_q)-(X+T+\gamma+15)}{2}.\]
     By direct computations, the statement follows.
\end{proof}

Now, we will focus on computing the maximal rates among the PIR schemes arising from hyperelliptic curves. We use the same notations as in Theorem \ref{thm pir rate hyperelliptic}.
As emphasized in \cite[Section VI.c]{costruzione2}, to maximize the PIR rate in Theorem \ref{thm pir rate hyperelliptic} one needs to compute the largest $N=N_{\max}^{\Yy}$  satisfying Theorem \ref{thm pir rate hyperelliptic}, i.e., such that there exist $N_{\max}^{\Yy}+g$ $\F_q$-rational points of $\mathcal{Y}$ disjoint from $\{P_\infty\} \cup \textnormal{supp}((y)_0) \cup \textnormal{supp}((h)_0). $ Being $z \mapsto \frac{z}{z+X+T+6g+2}$ an increasing function, this is equivalent to finding the largest $L=L_{\max}^{\Yy}$ (and cons. the largest $J=J_{\max}^{\Yy}$) such that $$N_{\max}^{\Yy}=L_{\max}^{\Yy}+X+T+6g+2=2J_{\max}^{\Yy}+X+T+5g+2.$$ Also, since $h =
\prod_{j\in[J]}(x-\lambda_j )$ depends on $L$, it is critical to maximize $L$ while minimizing the number of $\F_q$-rational points in $\textnormal{supp}((h)_0)$. Therefore, the choice of suitable $\lambda_1,\dots,\lambda_J \in \F_q$ is crucial and can be optimized as follows. Set  $\bar \Gamma:=\{x(P):P \in \Yy(\F_q)\}$ and  $\Gamma:=\F_q \setminus \bar \Gamma.$ Then \begin{equation}
    \#(\textnormal{supp}((x-\lambda_j)_0) \cap \F_q^2)=\begin{cases}
        0  & \mbox{ for } \lambda_j \in \Gamma \\ 
        1 \ \textnormal{or} \ 2 &  \mbox{ for } \lambda_j \in \bar \Gamma.
    \end{cases}
\end{equation}
This implies that selecting $\lambda_j \in \Gamma$ will not decrease the number of $\F_q$-rational points available, whereas choosing $\lambda_j \in \bar \Gamma$ will reduce the number by at most two. As each unit increase in 
$J$ leads to an increase of two in $L$, the optimal way to construct  $h$ is to select the $\lambda_j$'s by first choosing them from $\Gamma$ (as exhaustively as possible) and then from $\bar\Gamma$ until we cannot add more points without violating the condition on $N$ in Theorem \ref{thm pir rate hyperelliptic}. \\
In what follows, we will derive an explicit formula for $J_{\max}^{\Yy}$ over finite fields of odd characteristic. This allows us to compute $\Rr_{\max}^{\Yy}$ explicitly, since  $\Rr_{\max}^{\Yy}=\frac{2J_{\max}^{\Yy}-g}{2J_{\max}^{\Yy}+X+T+5g+2}.$\\
 Consider a hyperelliptic curve of genus $g$ as in Theorem \ref{thm pir rate hyperelliptic}. If $p$ is odd we can assume  without restriction that $H\equiv 0,$ that is $\Yy:y^2-f(x)=0,$ where  $\deg(f) = 2g+1$ and $f(x)=x^{2g+1}+a_{2g}x^{2g}+\dots+a_0.$ 
\begin{oss}
    Let $\bar \Gamma:=\{x(P):P \in \Yy(\F_q)\}$ and  $\Gamma:=\F_q \setminus \bar \Gamma.$ Then  \begin{equation} \label{eq formula Gamma}
        \#\Gamma=\dfrac{2q-\#\Yy(\F_q)-\gamma+1}{2}.
    \end{equation}
    where $0\leq \gamma:=
    \#\{z \in \F_q: f(z)=0\}\le 2g+1.$  Indeed it is easy so see that $\Yy(\F_q)=\{P_\infty\} \cup \{(z,0): z \in \F_q \ \textnormal{and} \ f(z)=0\} \cup \Sigma ,$ where $\Sigma=\{(\alpha,\beta): \alpha,\beta \in \F_q \ \textnormal{and} \ 0 \neq \beta^2 = f(\alpha)\}  $ is of even size. Thus $\#\bar \Gamma=\frac{\#\Yy(\F_q)-\gamma-1}{2}+ \gamma=\frac{\#\Yy(\F_q)+\gamma-1}{2}$ and Equation \eqref{eq formula Gamma} holds.
\end{oss}
We are now in position to prove the following result.
\begin{thm} \label{thm formula chiusa Jmax}
    Let $\Yy:y^2-x^{2q+1}-a_{2g}x^{2g}+\dots +a_0=0$  be a hyperelliptic curve of genus $g$ defined over $\F_q,$ $q$ odd. Let $J,L,N,X,T$ be as in Theorem \ref{thm pir rate hyperelliptic}. Then  \begin{equation}
    J_{\max}^\Yy=\begin{cases}
        \left\lfloor\frac{2q-(X+T+6g+2\gamma+2)}{4}\right\rfloor  & \mbox{ if } \#\Yy(\F_q) \ge q+3g+\frac{X+T+4}{2} \\ 
        \left\lfloor \frac{\#\Yy(\F_q)-(X+T+6g+3+\gamma)}{2} \right\rfloor &  \mbox{ otherwise }
    \end{cases}
\end{equation}
\end{thm}
\begin{proof}
We distinguish the following cases.\\
\textbf{Case} $J_{\max}^{\Yy} \ge \#\Gamma.$ 
\\
Set $J_{\max}^{\Yy}=\#\Gamma+\theta.$ Then $\#(\textnormal{supp}((h)_0)\cap \F_q^2)=2\theta$ and to maximize $J$ it is necessary to find the largest $\theta \ge 0$ such that \begin{eqnarray}
    \#\Yy(\F_q)\ge N+g+1+\gamma+2\theta=2\cdot \#\Gamma+X+T+6g+3+\gamma+4\theta,
\end{eqnarray} that is, by Equation \eqref{eq formula Gamma}, \begin{equation}\label{eq theta max Y}
    \theta = \left\lfloor \dfrac{2\cdot \#\Yy(\F_q)-(2q+X+T+6g+4)}{4} \right\rfloor.
\end{equation}
In particular $\theta \ge 0$ if and only if $\#\Yy(\F_q) \ge q+3g+\frac{X+T+4}{2}.$ Then by Equations \eqref{eq formula Gamma},\eqref{eq theta max Y} and direct computations, one gets $$J_{\max}^{\Yy}=\left\lfloor\frac{2q-(X+T+6g+2\gamma+2)}{4}\right\rfloor.$$
\textbf{Case} $J_{\max}^{\Yy} < \#\Gamma.$ 
\\
Now we can choose $\lambda_j \in \Gamma$ for each   $j\in [J_{\max}^{\Yy}].$ Then  $\textnormal{supp}((h)_0)\cap \F_q^2=\emptyset$ and $J_{\max}^{\Yy}$ is obtained by finding the largest $J$ such that 
\begin{equation} \label{eq J caso minore size Gamma}
    \#\Yy(\F_q)\ge N+g+1+\gamma=2J+X+T+6g+3+\gamma,
\end{equation} i.e. \begin{equation}\label{eq J for J<size Gamma}
    J_{\max}^{\Yy} = \left\lfloor \dfrac{\#\Yy(\F_q)-(X+T+6g+3+\gamma)}{2} \right\rfloor.
\end{equation}
 Equations $\eqref{eq formula Gamma}$ and \eqref{eq J caso minore size Gamma} yield $J_{\max}^{\Yy}< \#\Gamma$ for $\#\Yy(\F_q) < q+3g+\frac{X+T+4}{2}.$ The 
 statement follows.
    
\end{proof}
The following corollary is now immediate.
\begin{cor} \label{cor uppbound Jmax}
With the notations of Theorem  \ref{thm formula chiusa Jmax},   we have \begin{equation} \label{bound Jmax per gamma eq 0}
    J_{\max}^\Yy \le \left\lfloor\frac{2q-(X+T+6g+2)}{4}\right\rfloor.
\end{equation} In particular $\Rr_{\max}^{\Yy}=\frac{2J_{\max}^{\Yy}-g}{2J_{\max}^{\Yy}+X+T+5g+2}\le \frac{2q-(X+T+8g+2)}{2q+X+T+4g+2}.$
\end{cor}
\begin{oss} \label{rem hyp migliori sono quelle maxim/ottim}
From Theorem \ref{thm formula chiusa Jmax} we deduce that the best rate $\Rr_{\max}^\Yy$ is achieved by curves having enough points to satisfy the condition $\#\Yy(\F_q) \ge q+3g+\frac{X+T+4}{2}.$ This confirms that curves with many $\F_q$-rational points play a crucial role.    
\end{oss}
Define now $\Rr_{\max}^g$ as the maximum rate $\Rr_{\max}^\Yy$ with $\Yy$ ranging among all the hyperellptic curves of equation $y^2-x^{2g+1}-\sum _{i=0}^{2g}a_ix^i=0$.
 By Corollary \ref{cor uppbound Jmax} we have $$\Rr_{\max}^g \le  \frac{2q-(X+T+8g+2)}{2q+X+T+4g+2}.$$ Table 1 shows the values of $\Rr_{\max}^g$ for small values of $q$ and $g=1,2.$ Note that if $\Yy$ is an $\F_q$-optimal hyperelliptic curve, i.e., $\#\Yy(\F_q)=q+1+g(\Yy)\lfloor 2\sqrt{q}\rfloor,$ then by  \eqref{eq uppbound for hyp curves 2q+1} it follows that $ g(\Yy) \le \frac{\sqrt{q}}{2}.$ In particular $g(\Yy)\le 2$ for $q \le 29.$  
\begin{table}[h!]
\centering
\scalebox{0.6}{
\begin{tabular}{|c|c|c|c|c|c|c|c|c|c|c|c|c|c|c|c|}
\hline
$q$ & $g$ & $T=1$ & $T=2$ & $T=3$ & $T=4$  & $T=5$ & $T=6$ & $T=7$ & $T=8$ & $T=9$ & $T=10$ & $T=11$ & $T=12$ & $T=13$ & $T=14$\\
\hline
\multirow{2}*{$11$} & 1 & 0.33333 & 0.20000 & 0.066667 & - & - & - & - & - & - & - & - & - & - & -\\
& 2 & - & - & - & - & - & - & - & - & - & - & - & - & - & - \\ \hline
\multirow{2}*{$13$} & 1 & 0.41177 & 0.29412 & 0.26316 & 0.15789 & 0.052631 & - & - & - & - & - & - & - & - & -\\
& 2 & 0.11111 & - & - & - & - & - & - & - & - & - & - & - & - & - \\
\hline
\multirow{2}*{$17$} & 1 &0.52381 & 0.42857 & 0.39130 & 0.30435 & 0.21739 & 0.13043 & 0.043478 & - & - & - & - & - & - & -\\
& 2 & 0.27273 & 0.18182 & 0.16667 & 0.083333 & 0.076923 & - & - & - & - & - & - & - & - & - \\
\hline
\multirow{2}*{$19$} & 1 &0.56522 & 0.47826 & 0.44 & 0.36 & 0.28 & 0.2 & 0.12 & 0.04 & - & - & - & - & - & -\\
& 2 & 0.33333 & 0.25 & 0.23077 & 0.15385 & 0.14286 & 0.071428 & 0.066667 & - & - & - & - & - & - & - \\
\hline
\multirow{2}*{$23$} & 1 &0.62963 & 0.55556 & 0.51724 & 0.44828 & 0.41935 & 0.35484 & 0.29032 & 0.22581 & 0.16129 & 0.096774 & 0.032258 & - & - & -
\\
& 2 & 0.42857 & 0.35714 & 0.33333 & 0.26667 & 0.25 & 0.1875 & 0.17647 & 0.11765 & 0.11111 & 0.055555 & \textcolor{red}{0.052631} & - & - & -
 \\
\hline
\multirow{2}*{$25$} & 1 &0.65517 & 0.58621 & 0.54839 & 0.48387 & 0.45454 & 0.39394 & 0.33333 & 0.27273 & 0.21212 & 0.15152 & 0.090909 & 0.030303 & - & -
\\
& 2 & 0.46667 & 0.4 & 0.375 & 0.3125 & 0.29412 & 0.23529 & 0.22222 & 0.16667 & 0.15789 & 0.10526 & \textcolor{red}{0.1} & \textcolor{red}{0.05} & 0.047619 & -

 \\
\hline
\multirow{2}*{$27$} & 1 &0.67742 & 0.6129 & 0.57576 & 0.51515 & 0.48571 & 0.42857 & 0.37143 & 0.31429 & 0.25714 & 0.2 & 0.14286 & 0.085714 & 0.028571 & -

\\
& 2 & 0.5 & 0.4375 & 0.41176 & 0.35294 & 0.33333 & 0.27778 & 0.26316 & 0.21053 & 0.2 & 0.15 & 0.14286 & \textcolor{red}{0.095238} & \textcolor{red}{0.090909} & 0.045455

 \\
\hline
\multirow{2}*{$29$} & 1 &0.69697 & 0.63636 & 0.6 & 0.54286 & 0.51351 & 0.45946 & 0.40541 & 0.35135 & 0.2973 & 0.24324 & 0.18919 & 0.13514 & 0.081081 & 0.027027
\\
& 2 & 0.5 & 0.47059 & 0.41176 & 0.38889 & 0.33333 & 0.31579 & 0.26316 & 0.25 & 0.2 & 0.19048 & 0.14286 & \textcolor{red}{0.13636} & \textcolor{red}{0.090909} & -

 \\
\hline
\end{tabular}
}
\vspace{-0,2cm}
\caption{Values of $\Rr_{\max}^{g}$ over $\F_{q}$ for  $q \le 29,$ $g=1,2$ and $X=T$ }
\end{table}

\begin{oss}

    From Table 1, it can be seen that curves of genus $1$ yield better rates than those of genus $2$ in most cases (the exceptions are highlighted in red). Indeed Theorem \ref{thm formula chiusa Jmax} combined with the fact that $\F_q$-maximal curves exist
whenever $q$ is prime or a square of a prime power (see \cite[Theorem 4.3]{washington2008elliptic}) implies
    $\Rr^1_{max}>\Rr^2_{max}$  whenever $(2\sqrt{q}-3) \ge \frac{X+T+2}{2},$ i.e., $\sqrt{q}\ge \frac{X+T+14}{4}.$ Furthermore, even when the above condition is not satisfied, since the expressions for $J, L, N-L$ decrease linearly with $g$, better rates could be obtained from curves of genus $1$, especially for small values of $q$.
\end{oss}

We conclude this section comparing the rates of the schemes in the families (a)-(b)-(c), with those of the schemes in (d) constructed in Section \ref{sezionecostruzione}. To make such a comparison, we only consider finite fields of square order.

\begin{prop}
      Let $\Ee$ be an $\mathbb{F}_{q^2}$-maximal elliptic curve and $\Hh_q$ be the Hermitian curve of degree $q+1$. Then $\ \Rr_{\max}^{\mathcal{H}_q}> \Rr_{\max}^\mathcal{\Ee}$ for $q\ge 7$ and $X+T\geq 3(q+2).$ 
\end{prop}
\begin{proof}
By Proposition \ref{prop msx pir rate g eq 0,1} and the Hasse-Weil bound, one gets $L_{\max}^{\Ee}\leq \lfloor \frac{q^2+2q-X-T-10}{2} \rfloor,$  so \begin{equation}
    \Rr_{\max}^{\Ee} = \frac{L_{\max}^{\Ee}}{N_{\max}^{\Ee}} \ \le \frac{\frac{q^2+2q-X-T-10 }{2}}{\frac{q^2+2q-X-T-10 }{2} + X + T + 8} = \frac{q^2+2q-X-T-10}{q^2+2q+X+T+6} =: \bar \Rr_{\Ee}
\end{equation}
Then by Theorem \ref{thm PIR sheme for the hermitian} and Proposition \ref{prop rate migliore hermitiana} we have $L_{\max}^{\Hh_q}\geq \frac{q^3+1-(X+T+4q^2)}{2},$ thus
\begin{equation}\label{eq RH stima inferiore}
    \Rr_{\max}^{\Hh_q } \ge  \dfrac{\frac{q^3+1-(X+T+4q^2)}{2}}{\frac{q^3+1-(X+T+4q^2)}{2}+X+T+3q^2-q-2} =\dfrac{q^3+1-(X+T+4q^2)}{q^3+2q^2+X+T-(2q+3)}=:  \bar \Rr_{\Hh_q}
\end{equation}

Let us investigate for which values of $q,X,T$ we have $\bar \Rr_{\Hh_q} \ge \bar \Rr_{\Ee}$. By direct computations, $\bar \Rr_{\Hh_q} \ge \bar \Rr_{\Ee}$ is equivalent to 
$$-3q^4 + (X+T+3)q^3 - 2(X+T-2)q^2 - 3(X+T+2)q + (X+T-12) \ge 0 .$$

Define $M:=X+T$ and 
\begin{equation*}
   P(q,M):= -3q^4 + (M+3)q^3 - 2(M-2)q^2 - 3(M+2)q + (M-12). \end{equation*}
Then $\frac{\partial P(q,M)}{\partial M}=2q^3-4q^2-4q-4 > 0$ for $q\ge 3$ and $P(q,3q+6)=3q^3-17q^2-21q-6> 0$ for $q\geq 7.$ Thus $P(q,M) > 0$ for $M\ge 3q+6.$ The statement follows.
\end{proof}

In the following result, a comparison is made between  $\Rr_{\max}^{\Yy}$ and $\Rr_{\max}^{\Hh_q}.$ 
\begin{prop} \label{prop herm meglio di hyper}
    Let $\Yy:y^2-x^{2q+1}-a_{2g}x^{2g}+\dots +a_0=0$ be an hyperelliptic curve of genus $g$ and $\Hh_q$  be the Hermitian curve. Then $\ \Rr_{\max}^{\Hh_q}> \Rr_{\max}^\mathcal{\Yy}$ if one of the following conditions holds:
    \begin{itemize}
        \item $X+T\ge 3(2q+3),$ $g=1$ and $q>31;$
        \item $X+T\ge 3(2q+3),$ $g>1$ and $q>5.$
    \end{itemize} 
\end{prop}
\begin{proof}
    Define $\bar J_{\max}:=\left\lfloor\frac{2q^2-(X+T+6g+2)}{4}\right\rfloor.$ Then we have $$ \Rr_{\max}^\mathcal{\Yy} \le \frac{2\bar J_{\max}-g}{2\bar J_{\max}-g+X+T+6g+2}\le \frac{2q^2-(X+T+8g+2)}{2q^2+X+T+4g+2}=:\bar \Rr_{\Yy}.$$
    Let $\bar \Rr_{\Hh_q}$ be as in Equation \eqref{eq RH stima inferiore}. By easy computations we obtain that $\bar \Rr_{\Hh_q}> \bar \Rr_{\Yy}$ if and only if the inequality \begin{equation}
            -6q^4+(6g + X + T + 4)q^3-(3X + 3T - 2)q^2-(8g + X + T + 2)q+(2gX + 2gT - 10g - X - T - 2) > 0
    \end{equation} holds. Define $M:=X+T$ and \begin{equation} \label{eq P qMg}
        P(q,g,M):=-6q^4+(6g +M+ 4)q^3-(3M - 2)q^2-(8g + M + 2)q+(2gM - 10g - M - 2).
    \end{equation}
    Since $\frac{\partial P(q,g,M)}{\partial M}=q^3-3q^2-q+2g-1 > 0$ for $q>3$ and \begin{eqnarray}P(q,g,6q+9)&=&(6g+13)q^3-(18q+25)q^2-(8g+6q+11)q+(6q+9)(2g-1)-10g-2 \nonumber\\
    &=&(6g-5)q^3-31q^2-(20g+5)q+8g-11>0
\end{eqnarray} for either $g=1$ and $q>31,$ or $g>1$ and $q>5,$ the statement follows.
\end{proof}

Finally, we exhibit two tables where we compare explicitly the largest rates of XSTPIR schemes from $\F_{q^2}$-maximal hyperelliptic curves and from the Hermitian curve $\Hh_q$. \\In Table \ref{table hyp Tafazolian} we consider the hyperelliptic curve $\Yy:y^2-x^{2g+1}-1,$ which is $\F_{q^2}$-maximal, for $q$ odd, if and only if $2g+1 $ divides $ q+1$ (see \cite[Theorem 6]{tafazolian2012note}). In particular for $q=29,$ $\Yy$ is $\F_{q^2}$-maximal for $g \in \{1,2,7\}.$ Hence, the best possible rates are obtained for these values of $g$, according to Remark \ref{rem hyp migliori sono quelle maxim/ottim}. In accordance with Theorem \ref{thm formula chiusa Jmax} and Corollary \ref{cor uppbound Jmax} the highest rates (labelled in red) are obtained respectively for \begin{itemize}
    \item  $g=1,$ if $ 30 \le X+T \le 150;$  
    \item $g=2,$ if $150 \le X+T \le 360;$
    \item $g=7,$ if $X+T \ge 360.$
\end{itemize}
Finally, it turns out that $\Rr_{\max}^{\mathcal{Y}}>\Rr_{\max}^{\Xx}$ for $X+T\ge 30,$ in agreement with Proposition \ref{prop ellitt 1 costr meglio di raz}; indeed as pointed out in \cite{costruzione2}, the XSTPIR scheme costruction in Theorem \ref{thm pir rate hyperelliptic} generalizes for $g\ge 1$ and improves the rate of the  construction given in Theorem \ref{thm rate curva ellittica} for elliptic curves.

\begin{table}[h!] 
\centering
\scalebox{0.6}{
\begin{tabular} {|c|c|c|c|c|c|c|c|c|c|c|c|c|c|c|}
\hline
Genus & $T=15$ & $T=30$ & $T=45$ & $T=60$  & $T=75$ & $T=90$ & $T=105$ & $T=120$ & $T=135$ & $T=150$ & $T=165$ & $T=180$ & $T=195$ & $T=210$\\
\hline
0 & 0.93104 & 0.86667 & 0.80645 & 0.75000 & 0.69697 & 0.64706 & 0.60000 & 0.55556 & 0.51351 & 0.47368 & 0.43590 & 0.40000 & 0.36585 & 0.33333\\
1 & \textcolor{red}{0.95556} & \textcolor{red}{0.92193} & \textcolor{red}{0.88927} & \textcolor{red}{0.85699} & \textcolor{red}{0.82346} & 0.78995 & 0.75642 & 0.72291 & 0.68938 & 0.65587 & 0.62234 & 0.58883 & 0.55531 & 0.52179 \\
2 & 0.94860 & 0.91494 & 0.88262 & 0.85111 & 0.82096 & \textcolor{red}{0.79140} & \textcolor{red}{0.76321} & \textcolor{red}{0.73263} & \textcolor{red}{0.70105} & \textcolor{red}{0.66947} & \textcolor{red}{0.63789} & \textcolor{red}{0.60632} & 0.57474 & 0.54316\\
7 & 0.91345 & 0.88060 & 0.84859 & 0.81798 & 0.78798 & 0.75940 & 0.73122 & 0.70448 & 0.67795 & 0.65288 & 0.62786 & 0.60431 & \textcolor{red}{0.58067} & \textcolor{red}{0.55852} \\
\hline
\end{tabular}
}
\vspace{-0,2cm}
\caption{Values of $\Rr_{\max}^{\Xx}$ and   
  $\Rr_{\max}^{\mathcal{Y}}$ over $\F_{q^2}$ for $\mathcal{Y}:y^2=x^{2g+1}+1,$ $q=29,$ $g=1,2,7$ and $X=T$ }  \label{table hyp Tafazolian}
\end{table}

In Table \ref{tab confr hyp e herm} a comparison between $\Rr_{\max}^\Yy$ and $\Rr_{\max}^{\Hh_q},$ where $\Yy$ is an  $\F_{q^2}$-maximal hyperelliptic curve of genus $g$,  for $q=11$   is considered. Observe that only cases where $g\le 5$ are taken into account, in accordance with the bound \eqref{eq uppbound for hyp curves 2q+1}.
We observe that the costruction based on the Hermitian curve achieves a higher maximal PIR rate when $X+T$ is sufficiently large, in line with Proposition \ref{prop herm meglio di hyper}. The significantly larger number of $\F_{q^2}$-rational points on the Hermitian curve ($11^3+1=1332$) compared to the other curves (144, 166, 188, 210, and 232) provides greater flexibility in choosing the parameter $L$, leading to improved rates. This flexibility is further highlighted by the slow decrease of $\Rr_{\max}^{\Hh_q}$ as $X$ and $T$ increase (for $70 \le X+T \le 130$, it decreases by only $0.03836$), whereas $\Rr_{\max}^{\Yy}$ drops to around $0,18/0,19.$
\begin{table}[h!]
\centering
\scalebox{0.6}{
\begin{tabular}{|c|c|c|c|c|c|c|c|c|c|c|c|c|c|}
\hline
Genus & $T=5$ & $T=10$ & $T=15$ & $T=20$ & $T=25$ & $T=30$ & $T=35$ & $T=40$ & $T=45$ & $T=50$ & $T=55$ & $T=60$ & $T=65$ \\ \hline
1 & \textcolor{red}{0.86047} & \textcolor{red}{0.78947} & \textcolor{red}{0.72662} & \textcolor{red}{0.65958} & \textcolor{red}{0.58865} & 0.51773 & 0.44681 & 0.37589 & 0.30497 & 0.23404 & 0.16312 & 0.092198 & 0.021277 \\

2 & 0.81538 & 0.75000 & 0.68571 & 0.63013 & 0.57333 & \textcolor{red}{0.52564} & \textcolor{red}{0.47500} & 0.41975 & 0.35802 & 0.29630 & 0.23457 & 0.17284 & 0.11111 \\
3 & 0.77444 & 0.70803 & 0.65035 & 0.59184 & 0.54248 & 0.49044 & 0.44785 & 0.40120 & 0.36416 & 0.32203 & 0.28962 & 0.23497 & 0.18033 \\
4 & 0.73135 & 0.67142 & 0.61111 & 0.56000 & 0.50649 & 0.46250 & 0.41463 & 0.37647 & 0.33333 & 0.30000 & 0.26087 & 0.23158 & 0.19588 \\
5 & 0.70213 & 0.64138 & 0.58940 & 0.53548 & 0.49068 & 0.44242 & 0.40351 & 0.36000 & 0.32597 & 0.28649 & 0.25655 & 0.22051 & 0.19403 \\ \hline
$\frac{q(q-1)}{2}$ & 0.50890 & 0.49644 & 0.49061 & 0.47826 & 0.47271 & 0.46046 & 0.45517 & \textcolor{red}{0.44304} & \textcolor{red}{0.43800} & \textcolor{red}{0.43307} & \textcolor{red}{0.42117} & \textcolor{red}{0.41648} & \textcolor{red}{0.40468} \\
\hline
\end{tabular}
}
\vspace{-0,2cm}
\caption{Values of $\Rr_{\max}^{\mathcal{Y}}$ and $\Rr_{\max}^{\mathcal{H}_q}$ over $\F_{q^2}$ for $\mathcal{\Yy}:y^2=x^{2g+1}+\dots+a_0$ a $\F_{q^2}$-maximal \\ \centering hyperelliptic curve and $q=11,$ $g=1,2,3,4,5$, and $X=T$} \label{tab confr hyp e herm}
\end{table}

\section*{Conclusions}

In this work, we have constructed new PIR schemes based on algebraic geometry codes from the Hermitian curve, a famous example of maximal curve. By leveraging its large number of rational points, our approach enables longer code constructions and achieves higher retrieval rates than schemes based on genus 0, genus 1, and hyperelliptic curves of arbitrary genus. In particular, our results address one of the open directions outlined in \cite{costruzione2}, namely the exploration of PIR constructions from families of maximal curves.

Compared to existing CSA-based schemes, the Hermitian curve allows for a more flexible selection of parameters, especially in scenarios with fixed field size. This flexibility can translate into performance gains while avoiding certain limitations of genus-zero constructions, such as bounds on the number of servers relative to the field size. Moreover, our findings illustrate how maximal curves can serve as a natural source of efficient and high-rate PIR protocols.

As for future research, it would be interesting to extend our methods to other families of maximal curves, such as Suzuki or Ree curves, to further investigate the interplay between the genus, the number of rational points, and the achievable PIR rate. Another open question concerns identifying the most suitable algebraic curves for XSTPIR constructions within this framework: while maximal curves appear to be strong candidates, they are not always the optimal choice, and the criteria determining the best curves remain to be clarified.

\section*{Acknowledgments}
This work was funded by the SERICS project (PE00000014) under the MUR
National Recovery and Resilience Plan funded by the European Union-NextGenerationEU. 
The authors also
thank the Italian National Group for Algebraic and Geometric Structures and their Applications (GNSAGA-INdAM), which supported the research.

\bibliographystyle{plain}
\bibliography{References}

\end{document}